\documentclass[12pt, a4paper]{article}
\usepackage{amsmath,amsfonts}

\usepackage{amsfonts}
\usepackage{mathrsfs}
\usepackage{latexsym}
\usepackage{xy}
\usepackage{amsfonts,amsmath,amssymb,amsthm}
\usepackage{color}

\xyoption{all}

\newcommand{\bcen}{\begin{center}}     \newcommand{\ecen}{\end{center}}
\newcommand{\bay}{\begin{array}}      \newcommand{\eay}{\end{array}}
\newcommand{\beq}{\begin{eqnarray*}}      \newcommand{\eeq}{\end{eqnarray*}}

\def\rad{\mathrm{rad}}
\def\dim{\mathrm{dim}}

\def\Hom{\mathrm{Hom}}
\def\Ker{\mathrm{Ker}}
\def\Im{\mathrm{Im}}
\def\Ext{\mathrm{Ext}}
\def\HH{\mathrm{HH}}
\def\Tor{\mathrm{Tor}}

\def\mod{\mathrm{mod}}
\def\Mod{\mathrm{Mod}}

\def\pd{\mathrm{pd}}
\def\lim{\mathrm{lim}}

\def\per{\mathrm{per}}
\def\proj{\mathrm{proj}}
\def\Gproj{\mathrm{Gproj}}

\def\Cone{\mathrm{Cone}}

\def\thick{\mathrm{thick}}

\begin{document}

\newtheorem{theorem}{Theorem}[section]
\newtheorem{proposition}[theorem]{Proposition}
\newtheorem{lemma}[theorem]{Lemma}
\newtheorem{corollary}[theorem]{Corollary}
\newtheorem{remark}[theorem]{Remark}
\newtheorem{example}[theorem]{Example}
\newtheorem{definition}[theorem]{Definition}
\newtheorem{question}[theorem]{Question}
\numberwithin{equation}{section}

\title{\large\bf
Singular equivalences induced by ring extensions}

\author{\large Yongyun Qin}
\date{\footnotesize School of Mathematics, Yunnan Normal University, \\ Kunming, Yunnan 650500, China. E-mail:
qinyongyun2006@126.com
}

\maketitle

\begin{abstract} Let $B \subseteq A$ be an extension of finite
dimensional algebras.
We provide a sufficient condition for the existence
of triangle equivalences of singularity categories (resp. Gorenstein defect categories) between
$A$ and $B$.
This result is applied to trivial extensions, Morita rings and
triangular matrix algebras to give several reduction methods on
singularity categories and Gorenstein defect categories of algebras.

\end{abstract}

\medskip

{\footnotesize {\bf Mathematics Subject Classification (2020)}:
16E35; 16G10; 16E65; 18G80.}

\medskip

{\footnotesize {\bf Keywords}: Singularity categories;
Gorenstein defect categories; Extensions of algebras. }

\bigskip

\section{\large Introduction}

\indent\indent
Throughout all algebras are finite dimensional associative
algebras over a field $k$, and all modules are finitely generated left modules unless stated
otherwise. The {\it singularity category} $D_{sg}(A)$ of an algebra
$A$ is defined as the Verdier quotient of the bounded derived category of finitely generated modules over $A$ by the
full subcategory of perfect complexes \cite{Buch21}.
It measures the ``regularity'' of $A$ in the sense
that $D_{sg}(A)=0$ if and only if $A$ has finite global
dimension. According to \cite{Buch21}, there is an embedding functor
$F$ from the stable category $\underline{\Gproj} A$ of finitely generated Gorenstein projective
modules to $D_{sg}(A)$, and the {\it Gorenstein
defect category} of $A$ is defined to be the Verdier quotient
$D_{def}(A):=D_{sg}(A)/ \Im F$, see \cite{BJO15}. This category
measures how far the algebra $A$ is from
being Gorenstein in the sense that
$D_{def}(A)=0$ if and only if $A$ is Gorenstein \cite{BJO15}.
In general, it is difficult to describe the singularity categories
and Gorenstein defect categories. Many people described these categories for
some special kinds of algebras \cite{Chen09, Chen18, CGL15, CL15, CSZ18, Kal15, Rin13},
and some experts compared these categories
between two algebras related to one another \cite{Chen14, EPS22, LHZ22, Lu17, Lu19,
PSS14, Qin24, QS23, Shen21, Z13}.

An {\it extension} of algebras is simply a finite dimensional algebra $A$ with an unital
subalgebra $B$, denoted $B \subseteq A$, and it provide us a useful framework
to reduce
the homological properties of algebras. Indeed, there are numerous works
comparing the
homological properties of $A$ and $B$ under certain conditions \cite{CLMS22, IM21, MN24, Xi06, XX13}.
From \cite{CLMS22}, an extension $B \subseteq A$ is called
{\it left (resp. right) bounded} if $A/B$ has finite projective
dimension as a $B$-bimodule, $A/B$ is projective as a left (resp. right) $B$-module and
some tensor power $(A/B)^{\otimes _Bp}$
is 0. Using the Jacobi-Zariski long exact sequence, Cibils et al. proved that
for a bounded extension $B \subseteq A$,
the Hochschild
homology groups of $B$ vanish in high enough degree if and only if so does $A$,
and thus $B$ satisfies Han's conjecture if and only if $A$ does \cite{CLMS22}.
Moreover, Iusenkoa and MacQuarrie showed the finitistic
dimension conjecture is preserved under bounded extensions \cite{IM21}.
On the other hand, both the Hochschild homology and the finitistic
dimension conjecture are invariant under singular equivalence of Morita type with level \cite{Wang15}.
So it is nature to compared the singularity categories
between two algebras linked by a bounded extension.
Our main result is the following, which is listed in Theorem~\ref{theorem-sing-equiv}
and Theorem~\ref{theorem-def-equiv}.

\begin{theorem}\label{theorem-main}
Let $A$ and $B$
be two algebras and $B\subseteq A$ be an extension. Assume that
$\pd _{B^e}(A/B)<\infty$, $(A/B)^{\otimes _Bp}=0$ for some integer $p$ and
$\Tor _i^B(A/B, (A/B)^{\otimes _Bj})=0$ for each $i,j\geq 1$.
Then there is a
singular equivalence of Morita type with level between $A$ and $B$. Moreover,
this singular equivalence induces an equivalence on the Gorenstein defect categories
if the extension is split.
\end{theorem}

Here, an extension $B\subseteq A$ is split if the algebra inclusion $B\hookrightarrow A$
splits. From \cite{CL20}, the Tor-vanishing condition in Theorem~\ref{theorem-main}
is equivalent to $\Tor _i^B((A/B)^{\otimes _Bj}, A/B)=0$ for each $i,j\geq 1$.
Therefore, the assumption of Theorem~\ref{theorem-main} is weaker than the definition
of bounded extension,
where the Tor-vanishing condition is replaced
by $A/B$ is projective as a left or a right $B$-module.
As a result, we generalize the main result of
Cibils et al. \cite{CLMS22} by showing that Han's conjecture is preserved
under certain extensions which are strictly bigger than
bounded extensions, see Remark~\ref{rem-ext} and Corollary~\ref{cor-Han-conj}.
Moreover, we apply Theorem~\ref{theorem-main} to trivial extensions, Morita rings
and triangular matrix algebras, and we get several reduction methods on
singularity categories and Gorenstein defect categories,
see Corollary~\ref{cor-trivial-extension}, Corollary~\ref{cor-tri-matrix-alg},
Corollary~\ref{cor-Morita-contex} and Example~\ref{exam-1}.

This paper is organized as follows. In section 2, we recall some relevant
definitions and conventions.
In section 3 we prove Theorem I and we generalize the result of
Cibils et al. on Han's conjecture.
Some applications and an example are given in the final section.

 \section{\large Definitions and conventions}\label{Section-definitions and conventions}

\indent\indent In this section we will fix our notations and recall some basic definitions.

Let $\mathcal{S}$ be a set of objects of
a triangulated category $\mathcal{T}$. We denote by $\thick \mathcal{S}$
the smallest triangulated subcategory of $\mathcal{T}$
containing $\mathcal{S}$ and closed under taking direct summands.

Let $A$ be finite dimensional associative
algebras over a field $k$. We denote by $\Mod A$ the
category of left $A$-modules, and we view right $A$-modules
as left $A^{op}$-modules, where $A^{op}$ is the opposite algebra of $A$.
Denote by $\mod A$ and $\proj A$
the full subcategories of $\Mod A$ consisting of all finitely
generated modules and finitely
generated projective modules, respectively.

Let $K (A)$ be the homotopy category
of complexes over $\Mod A$, and
let
$K^b (\proj A)$ be the bounded homotopy category of complexes
over $\proj A$. We denote by
$\mathcal{D}(\Mod A)$ (resp. $\mathcal{D}^b(\mod A)$)
the derived category (resp. bounded derived category) of complexes over $\Mod A$ (resp. $\mod A$),
and denote by $[-]$ the shift functor on complexes.
Usually, we
just write $\mathcal{D} A$ (resp. $\mathcal{D}^b(A)$) instead of $ \mathcal{D}(\Mod A)$ (resp. $ \mathcal{D}^b(\mod A)$).
Up to isomorphism, the objects in $K^{b}(\proj A)$ are
precisely all the compact objects in $\mathcal{D} A$. For
convenience, we do not distinguish $K^{b}(\proj A)$ from the {\it
perfect derived category} $\mathcal{D}_{\per}(A)$ of $A$, i.e., the
full triangulated subcategory of $\mathcal{D} A$ consisting of all
compact objects, which will not cause any confusion. Moreover, we
also do not distinguish
$\mathcal{D}^b(A)$ from its essential image under the
canonical embedding into $\mathcal{D} A$.

Following \cite{Buch21, Orl04}, the {\it singularity category} of $A$ is the
Verdier quotient $D_{sg}(A) = \mathcal{D}^b(A)/K^{b}(\proj A)$.
A finitely
generated $A$-module $M$ is called
{\it Gorenstein projective} if there is an
exact sequence $$\xymatrix{P^\bullet = \cdots \ar[r]& P^{-1}
\ar[r]^{d^{-1}} & P^0 \ar[r]^{d^{0}} & P^1
\ar[r] & \cdots} $$
of $\proj A$ with $M= \Ker d^0$ such that $\Hom _A(P^\bullet, Q)$ is exact for
every $Q \in \proj A$. Denote by $\Gproj A$
the subcategory of $\mod A$ consisting of Gorenstein projective modules.
It is well known that $\Gproj A$
is a Frobenius category, and hence its stable category $\underline{\Gproj} A$ is a triangulated category.
Moreover, there is a canonical triangle functor $F: \underline{\Gproj} A \rightarrow D_{sg}(A)$
sending a Gorenstein projective module to the corresponding stalk complex concentrated in degree zero
\cite{Buch21}.

\begin{definition}{\rm (See \cite{BJO15})
The Verdier quotient $D_{def}(A):=D_{sg}(A)/ \Im F$
is called the} Gorenstein defect
category {\rm of $A$}.
\end{definition}

For an algebra $A$, we denote by $\Omega _A (-)$ the syzygy functor on the
stable category $\underline{\mod }A $ of $A$-modules, and
define
${}^{\perp}A:=\{X\in\mod A|\Ext_A^i(X, A)=0\mbox{ for all }i>0\}$.
Denote by $D^b(A)_{fGd}$ the full subcategory of
$D^b(A)$ formed by those complexes quasi-isomorphic to bounded complexes of Gorenstein projective modules.
Here, the definition of $D^b(A)_{fGd}$ agrees with that in \cite{Kato02},
where the objects in $D^b(A)_{fGd}$ are called complexes of finite Gorenstein projective dimension,
see \cite[Definition 2.7 and Proposition 2.10]{Kato02}.
 Moreover, $D^b(A)_{fGd}$ is a thick subcategory
of $D^b(A)$ generated by all the Gorenstein projective modules, that is,
$D^b(A)_{fGd}=\thick (\Gproj A)$, see \cite[Theorem 2.7]{LHZ22}.
The following equivalence
$$ \underline{\Gproj }A \cong D^b(A)_{fGd}/K^b(\proj A)$$
is well known, see \cite[Theorem 4.4.1]{Buch21}, \cite[Lemma 4.1]{CR20}
or \cite[Theorem 4]{PZ15} for examples.

\section{Singularity categories and ring extensions}
\indent\indent In this section, we will compare the singularity categories and the Gorenstein
defect categories between two algebras linked by an extension.
The following lemma is well-known. Here, we give a proof for readers' convenience.
\begin{lemma}\label{lem-tensor}
Let $A$, $B$ and $C$ be algebras, $M$ be an $A$-$B$ bimodule and $N$ be a $B$-$C$ bimodule.
If $\Tor _i^B(M,N)=0$ for any $i>0$, then $M\otimes _B^LN\cong M\otimes _BN\cong$ in $\mathcal{D}(A\otimes C^{op})$.

\end{lemma}

\begin{proof}
Let $$\cdots \longrightarrow P^{-2} \longrightarrow
P^{-1} \longrightarrow P^{0} \longrightarrow
M \longrightarrow 0 $$
be the projective resolution of $M$ as an $A$-$B$ bimodule
and let $P^\bullet$ be the deleted complex $$ \cdots \longrightarrow P^{-2} \longrightarrow
P^{-1} \longrightarrow P^{0} \longrightarrow 0.$$ Then
$M\otimes _B^LN\cong P^\bullet\otimes _BN$ in $\mathcal{D}(A\otimes C^{op})$. Since $\Tor _i^B(M,N)=0$ for any $i>0$, the complex
 of $A$-$C$ bimodules $$\cdots \longrightarrow P^{-2}\otimes _BN \longrightarrow
P^{-1}\otimes _BN \longrightarrow P^{0}\otimes _BN \longrightarrow
M\otimes _BN \longrightarrow 0 $$ is exact, and thus $M\otimes _BN\cong P^\bullet\otimes _BN$ in $\mathcal{D}(A\otimes C^{op})$.
Therefore, $M\otimes _B^LN\cong M\otimes _BN\cong$ in $\mathcal{D}(A\otimes C^{op})$.
\end{proof}

The following lemma is essentially due to \cite[Lemma 2.5]{AKLY17}.
\begin{lemma}\label{lem-preserve-kbproj}
Let $B\subseteq A$ be a ring extension.
Then the following statements hold:

{\rm (1)} The functor $A\otimes _B^L-\otimes _B^LA :\mathcal{D}(B^e)\rightarrow \mathcal{D}(A^e)$
sends $K^b(\proj B^e)$ to $K^b(\proj A^e)$.

{\rm (2)} If $\pd _{B^e}(A/B)<\infty$, then the functor $_B(A/B)\otimes _B^L-:\mathcal{D}(B^e)\rightarrow \mathcal{D}(B^e)$
sends $K^b(\proj B^e)$ to $K^b(\proj B^e)$.

\end{lemma}

\begin{proof}
(1) This follows from the isomorphisms $$A\otimes _B^L(B\otimes _kB)\otimes _B^LA
\cong (A\otimes _B^LB)\otimes _k(B\otimes _B^LA)\cong A\otimes _kA$$ and
\cite[Lemma 2.5]{AKLY17}.

(2) Let $P^\bullet \rightarrow A/B$ be the bounded projective resolution of $A/B$ as a $B$-$B$ bimodule.
Then $_B(A/B)\otimes _B^L(B\otimes _kB)\cong P^\bullet \otimes _B(B\otimes _kB)$ in $\mathcal{D}(B^e)$.
Since $(B\otimes _kB)\otimes _B(B\otimes _kB)\cong B\otimes _kB\otimes _kB\cong (B\otimes _kB)^{\dim _kB}$,
we have that $P^i \otimes _B(B\otimes _kB)\in \proj B^e$. Therefore, $P^\bullet \otimes _B(B\otimes _kB)
\in K^b(\proj B^e)$ and then the statement follows from \cite[Lemma 2.5]{AKLY17}.
\end{proof}

\begin{lemma}\label{lem-ring-extension-tor}
Let $B\subseteq A$ be a ring extension.
If $\Tor _i^B(A/B, (A/B)^{\otimes _Bj})=0$ for each $i,j\geq 1$, then the following statements hold:

{\rm (1)} $\Tor _i^B(A, A)=0$ for each $i\geq 1$;

{\rm (2)} $\Tor _i^B((A/B)^{\otimes _Bj}, A)=0$ for each $i,j\geq 1$;

{\rm (3)} $\Tor _i^B(A, (A/B)^{\otimes _Bj}\otimes _BA)=0$ for each $i,j\geq 1$.
\end{lemma}

\begin{proof}
(1) For any $i\geq 1$, applying $A/B \otimes _B-$ to the exact sequence of left $B$ modules
$$0\rightarrow B\rightarrow A\rightarrow A/B\rightarrow 0$$
and using the fact that $\Tor _i^B(A/B, A/B)=0$,
we have that
 $\Tor _i^B(A/B, A)=0$. Similarly, applying $-\otimes _BA$ to the exact sequence of right $B$ modules
$$0\rightarrow B\rightarrow A\rightarrow A/B\rightarrow 0,$$ we get $\Tor _i^B(A, A)=0$ for each $i\geq 1$.

(2) It follows from \cite[Corollary 4.3]{CL20} that $\Tor _i^B((A/B)^{\otimes _Bj},A/B)=0$ for each $i,j\geq 1$.
Applying $(A/B)^{\otimes _Bj}\otimes _B-$ to the exact sequence of left $B$ modules
$$0\rightarrow B\rightarrow A\rightarrow A/B\rightarrow 0,$$ we infer that $\Tor _i^B((A/B)^{\otimes _Bj}, A)=0$ for each $i,j\geq 1$.

(3) For each $i,j\geq 1$, we have that $\Tor _i^B(A/B, (A/B)^{\otimes _Bj}\otimes _BA)=0$
by (2) and \cite[Lemma 4.2]{CL20}.
Applying $-\otimes _B(A/B)^{\otimes _Bj}\otimes _BA$ to the exact sequence of right $B$ modules
$$0\rightarrow B\rightarrow A\rightarrow A/B\rightarrow 0,$$ we infer that $\Tor _i^B(A, (A/B)^{\otimes _Bj}\otimes _BA)=0$ for each $i,j\geq 1$.
\end{proof}

Now we will give a sufficient condition for the existence of a singular equivalence
between two algebras linked by a ring extension.
This constitutes the first part of the main theorem presented in the
introduction.
\begin{theorem}\label{theorem-sing-equiv}
Let $A$ and $B$
be two algebras and $B\subseteq A$ be a ring extension. Assume that
$\pd _{B^e}(A/B)<\infty$, $(A/B)^{\otimes _Bp}=0$ for some integer $p$ and
$\Tor _i^B(A/B, (A/B)^{\otimes _Bj})=0$ for each $i,j\geq 1$.
Then there is a
singular equivalence of Morita type with level between $A$ and $B$.
\end{theorem}

\begin{proof}
Consider the adjoint triple between derived categories
$$\xymatrix@!=8pc{ \mathcal{D}A \ar[r]|{G=_BA\otimes _A-} & \mathcal{D}B \ar@<-3ex>[l]_{F=A\otimes _B^L-}
\ar@<+3ex>[l] }.$$ Since $\pd _{B^e}(A/B)<\infty$, we get that $\pd (A/B)_B<\infty$ and $\pd _B(A/B)<\infty$.
Now the exact sequence of $B$-$B$ bimodules
$$0\rightarrow B\rightarrow A\rightarrow A/B\rightarrow 0$$ shows that $\pd A_B<\infty$ and $\pd _BA<\infty$.
It follows from \cite[Lemma 2.8]{AKLY17} that both $F$ and $G$ restricts to $\mathcal{D}^b(\mod )$,
and by \cite[Lemma 2.5]{AKLY17}, both $F$ and $G$ restricts to $K^b(\proj )$.
Therefore, the adjoint functors $F$ and $G$ restrict to one at the level
of singularity categories, see \cite[Lemma 1.2]{Orl04}. Now we claim that
the natural morphisms $B\rightarrow GF(B)$ and $FG(A)\rightarrow A$
are isomorphisms in $D_{sg} (B^e)$ and $D_{sg} (A^e)$, and then there is a
singular equivalence of Morita type with level between $A$ and $B$,
see \cite[Theorem 3.6]{Dal21} and \cite[Proposition 3.3]{Qin22}.

To see the claim, we observe that the mapping cone of $B\rightarrow GF(B)$ is $A/B$, and
the assumption
$\pd _{B^e}(A/B)<\infty$ yields that $B\cong GF(B)$ in $D_{sg} (B^e)$.
By Lemma~\ref{lem-ring-extension-tor} (1) and Lemma~\ref{lem-tensor},
we have that
$FG(A)=A\otimes _B^LA\cong A\otimes _BA$ in $\mathcal{D}(A^e)$,
and then mapping cone of $FG(A)\rightarrow A$ is isomorphic
to the complex $0\rightarrow A\otimes _BA\rightarrow A\rightarrow 0$
induced by the product of $A$. Since $(A/B)^{\otimes _Bp}=0$ for some integer $p$, it follows from \cite[Proposition 2.3]{CLMS21}
that there is a long exact sequence
of $A$-$A$ bimodules
$$\xymatrix@C=0.8pc{ 0 \ar[r] & A\otimes _B(A/B)^{\otimes _B{(p-1)}}\otimes _BA \ar[r] & \cdots
\ar[r] &  A\otimes _B(A/B)\otimes _BA \ar[r] &  A\otimes _BA \ar[r] & A \ar[r] & 0,}$$
where the last non-zero differential is the product of $A$.
Therefore, the nature chain map
$$\xymatrix@C=0.8pc{ 0 \ar[r] & A\otimes _B(A/B)^{\otimes _B{(p-1)}}\otimes _BA \ar[r] & \cdots
\ar[r] &  A\otimes _B(A/B)\otimes _BA \ar[r] \ar[d]& 0 \ar[d] \ar[r] &0\\
& & 0 \ar[r] & A\otimes _BA \ar[r] & A \ar[r] &0}$$
is a quasi-isomorphism, and thus the mapping cone of $FG(A)\rightarrow A$ is isomorphic
to the complex
$$\xymatrix@C=0.8pc{ 0 \ar[r] & A\otimes _B(A/B)^{\otimes _B{(p-1)}}\otimes _BA \ar[r] & \cdots
\ar[r] &  A\otimes _B(A/B)\otimes _BA \ar[r] & 0}$$
in $\mathcal{D}(A^e)$. Now we claim that $A\otimes _B(A/B)^{\otimes _Bj}\otimes _BA \in
K^b(\proj A^e)$ for any $j\geq 1$, and thus $FG(A)\cong A$
in $D_{sg} (A^e)$. Since $\Tor _i^B(A/B, (A/B)^{\otimes _Bj})=0$ for each $i,j\geq 1$,
it follows from Lemma~\ref{lem-tensor} that $(A/B)\otimes _B
(A/B)\otimes _B\cdots \otimes _B
(A/B)\cong (A/B)\otimes _B^L
(A/B)\otimes _B^L\cdots \otimes _B^L
(A/B)$ in $\mathcal{D}(B^e)$, which belongs to $K^b(\proj B^e)$ by Lemma~\ref{lem-preserve-kbproj} (2).
Therefore, $(A/B)^{\otimes _Bj}\in K^b(\proj B^e)$ for any $j\geq 1$,
and it follows from Lemma~\ref{lem-preserve-kbproj} (1) that $A\otimes _B^L(A/B)^{\otimes _Bj}\otimes _B^LA \in
K^b(\proj A^e)$. On the other hand, Lemma~\ref{lem-ring-extension-tor} (2), Lemma~\ref{lem-ring-extension-tor} (3)
and Lemma~\ref{lem-tensor} yield that $A\otimes _B^L(A/B)^{\otimes _Bj}\otimes _B^LA
\cong A\otimes _B(A/B)^{\otimes _Bj}\otimes _BA$. Therefore, $A\otimes _B(A/B)^{\otimes _Bj}\otimes _BA \in
K^b(\proj A^e)$ for any $j\geq 1$.
\end{proof}

\begin{remark}\label{rem-ext}
{\rm By \cite[Corollary 4.3]{CL20}, the Tor-vanishing condition in Theorem~\ref{theorem-sing-equiv}
is equivalent to $\Tor _i^B((A/B)^{\otimes _Bj}, A/B)=0$ for each $i,j\geq 1$.
Therefore, the extensions considered in Theorem~\ref{theorem-sing-equiv} are strictly bigger than
bounded extensions. For example, let $A=
\left(\begin{array}{cc} B & 0 \\ M & C  \end{array}\right)$ be a triangular matrix algebra
with $M$ a $C$-$B$-bimodule. Then $A\cong (B\times C)\ltimes M$ and there is an extension
$B\times C\hookrightarrow A$ with $A/(B\times C)=M$.
Since $M\mathop{\otimes}\limits_{B\times C}M=0$
and $\Tor _i^{B\times C}(M, M)=0$ for $i\geq 1$, it follows that the extension
$B\times C\hookrightarrow A$ satisfies all conditions in Theorem~\ref{theorem-sing-equiv}
if $\pd _{C\otimes B^{op}}M<\infty$. However, this extension is not bounded
unless $M$ is projective as a left $C$-module or a right $B$-module.}
\end{remark}

Recall that Han's conjecture \cite[Conjecture 1]{Han06}
asserts that the $n$-th Hochschild homology $\HH_n(A)$ vanishes for sufficiently
large $n$ if and only if $A$ has finite global dimension.
In \cite{CLMS22}, Cibils et al. proved that Han's conjecture is preserved
under bounded extensions. Now we
will generalize their result as follow.

\begin{corollary}\label{cor-Han-conj}
Let $A$ and $B$
be two algebras and $B\subseteq A$ be a ring extension. Assume that
$\pd _{B^e}(A/B)<\infty$, $(A/B)^{\otimes _Bp}=0$ for some integer $p$ and
$\Tor _i^B(A/B, (A/B)^{\otimes _Bj})=0$ for each $i,j\geq 1$.
Then $A$ satisfies Han's conjecture if and only if $B$ does.
\end{corollary}
\begin{proof}
By Theorem~\ref{theorem-sing-equiv}, there is a singular
equivalence of Morita type with level between $A$ and $B$.
It follows from \cite[Proposition 3.8]{Wang15} that $\HH_i(A)\cong
\HH_i(B)$ for any $i>0$. Since $A$ and $B$ are singularly equivalent,
we have that $A$ has finite global dimension if and only if $B$ does.
Therefore, we infer that $A$ satisfies Han's conjecture if and only if $B$ does.
\end{proof}

Next, we will compare the Gorenstein defect categories
between two algebras linked by a ring extension. We start with the following lemma.
\begin{lemma}\label{lemma-restr-Gproj}
Let $G:\mod A\rightarrow \mod B$ be an exact functor with $\pd _BG(A)< \infty$.
Assume there is an integer $t$ such that $\Omega ^t(GX)\in {}^{\perp} B$ for any $X\in {}^{\perp}A$. Then
there is an integer $l$ such that $\Omega ^l(GX)\in \Gproj B$ for any $X\in \Gproj A$, and
$G$ induces a triangle functor from $D^b(A)_{fGd}$ to $D^b(B)_{fGd}$.
\end{lemma}
\begin{proof}
This can be proved similarly as \cite[Lemma 3.8]{QS23}.
\end{proof}

Now we will prove the second part of the main theorem presented in the introduction.
\begin{theorem}\label{theorem-def-equiv}
Let $B\subseteq A$ be a split extension with all the conditions in Theorem~\ref{theorem-sing-equiv}
satisfied.
Then the
singular equivalence in Theorem~\ref{theorem-sing-equiv} induces an equivalence on the Gorenstein defect categories.
\end{theorem}

\begin{proof}
It is easy to see that the functor $G={_BA}\otimes _A-: \mod A\rightarrow \mod B$ is exact and
$\pd _BG(A)< \infty$.
Now we claim that there is an integer $t$ such that $\Omega ^t(GX)\in {}^{\perp} B$ for any $X\in {}^{\perp}A$,
and then Lemma~\ref{lemma-restr-Gproj} can be applied.

Since the extension is split, we have that $_BB$ is a direct summand of $_BA$.
For any $X\in {}^{\perp}A$ and $i>0$, consider the isomorphism $$\Ext _B^i(GX,B)\cong
\Hom _{\mathcal{D}B}(GX, B[i]),$$
which is a direct summand of $\Hom _{\mathcal{D}B}(GX, GA[i])$.
By adjointness, we have that $$\Hom _{\mathcal{D}B}(GX, GA[i])\cong \Hom _{\mathcal{D}A}(FGX, A[i]).$$
Let $\eta : FG\rightarrow 1_{\mathcal{D}A}$ be the counit of the adjoint pair
$$\xymatrix{ \mathcal{D}B \ar@<+1ex>[rr]^{F}
&&\mathcal{D}A\ar@<+1ex>[ll]^{G}}.$$
Then there is a canonical triangle
$$\xymatrix{FGX \ar[r]^{\eta _X}  & X
 \ar[r] & \Cone (\eta _X) \ar[r] & }$$
in $\mathcal{D}A$.
Applying $\Hom _{\mathcal{D}A}(-, A[i])$, we get
a long exact
sequence
$$\xymatrix@C=0.5pc{\cdots \ar[r] & \Hom _{\mathcal{D}A}(\Cone (\eta _X), A[i]) \ar[r] &
\Hom _{\mathcal{D}A}(X, A[i]) \ar[r] & \Hom _{\mathcal{D}A}(FGX, A[i])\ar[r] & \cdots  (\star).
}$$
Since $FGX\cong FGA\otimes _AX$, it follows from the commutative diagram
 $$\xymatrix@C=4pc{FGX \ar[r]^{\eta _X} \ar[d]^{\cong} & X
\ar[d]^{\cong} \ar[r] & \Cone (\eta _X) \ar[r]\ar[d] &
\\ FGA\otimes _AX \ar[r]^{\eta _A\otimes _AX }  & A\otimes _AX \ar[r]  & \Cone (\eta _A)\otimes _AX \ar[r] & }$$
that $\Cone (\eta _X)\cong \Cone (\eta _A)\otimes _AX $ in $\mathcal{D}A$.
According to the proof of Theorem~\ref{theorem-sing-equiv}, we see that $\Cone (\eta _A)\in K^b(\proj A^e)$.
Then there is a bounded complex of projective $A$-$A$ bimodules
$$P^\bullet: 0\rightarrow P^{-t}\longrightarrow \cdots
\longrightarrow P^{s} \longrightarrow
0 $$
such that $P^\bullet \cong \Cone (\eta _A)$ in $\mathcal{D}(A^e)$.
Therefore, the complex $\Cone (\eta _X)\cong P^\bullet \otimes _AX $
is of the form $$0\rightarrow P^{-t}\otimes _AX \longrightarrow \cdots
\longrightarrow P^{s}\otimes _AX \longrightarrow
0 $$ where each $_AP^{i}\otimes _AX\in \proj A$ since $P^{i}\in \proj A^e$.
It follows that $$\Hom _{\mathcal{D}A}(\Cone (\eta _X), A[i])\cong \Hom _{K(A)}(P^\bullet \otimes _AX, A[i])$$
which is equal to zero for any $i>t$. Therefore, the long exact sequence $(\star)$
yields that $\Hom _{\mathcal{D}A}(FGX, A[i]) \cong \Hom _{\mathcal{D}A}(X, A[i]) $
for any $i>t$. On the other hand, $\Hom _{\mathcal{D}A}(X, A[i])\cong
\Ext_A^i(X,A)=0$ for any $i>0$ since $X\in {}^{\perp}A$.
As a result, we get that $\Hom _{\mathcal{D}A}(FGX, A[i])=0$ for any $i>t$,
and thus $\Ext _B^i(GX,B)=0$ for any $i>t$. Therefore, $\Omega ^t(GX)\in {}^{\perp} B$ for any $X\in {}^{\perp}A$
and by Lemma~\ref{lemma-restr-Gproj}, $G$ induces a triangle functor from $D^b(A)_{fGd}$ to $D^b(B)_{fGd}$.

On the other hand, it follows from \cite[Lemma 2.8]{AKLY17} that $F=A\otimes _B^L-$
has a left adjoint which restricts to $K^b(\proj )$, and by \cite[Lemma 3.4]{CHQW20}, $F$ restricts to
a non-negative functor from $\mathcal{D}^b(B)$ to $\mathcal{D}^b(A)$, up to shifts.
By \cite[Proposition 5.2]{HP17},
the stable functor $\overline{F}$ preserves Gorenstein projective
modules, and according to \cite[Section 4.2]{HP17},
each $X\in \mod B$ yields a triangle $$P^\bullet _X\rightarrow F(X)\rightarrow
\overline{F}(X)\rightarrow$$ in $\mathcal{D}A$ with $P^\bullet _X\in K^b(\proj A)$.
Therefore, $FX \in D^b(A)_{fGd}$ for any $X\in \Gproj B$, and then
$F$ sends the objects of $D^b(B)_{fGd}$ to $D^b(A)_{fGd}$. Therefore, $F$ and $G$ induce an adjoint pair between
$D^b(B)_{fGd}/ K^b(\proj B)$ and $D^b(A)_{fGd}/ K^b(\proj A)$, see \cite[Lemma 1.2]{Orl04}.
Thanks to the equivalence
$$ \underline{\Gproj }A \cong D^b(A)_{fGd}/K^b(\proj A),$$ we obtain
an adjoint pair
$$\xymatrix{ \underline{\Gproj} B\ar@<+1ex>[rr]^{\widetilde{F}}
&&\underline{\Gproj} A\ar@<+1ex>[ll]^{\widetilde{G}}},$$
and combining Theorem~\ref{theorem-sing-equiv}, we have the following exact commutative diagram
$$\xymatrix{0 \ar[r] &\underline{\Gproj} A \ar[r] \ar@<-1ex>[d]_ {\widetilde{G}} & D_{sg}(A)
\ar@<-1ex>[d]_G \ar[r] & D_{def}(A) \ar[r] & 0
\\ 0 \ar[r] &\underline{\Gproj} B \ar[r] \ar@<-1ex>[u]_{\widetilde{F}} & D_{sg}(B) \ar[r] \ar@<-1ex>[u]_{F} & D_{def}(B) \ar[r] & 0 }, $$
where the vertical functors between $D_{sg}(A)$ and $D_{sg}(B)$ are equivalences.
Since $G$ is fully faithful, we infer that $\widetilde{G}$ is also fully faithful
from the above exact commutative diagram. Now we claim that $\widetilde{G}$ is dense
and then $\widetilde{G}: \underline{\Gproj }A\rightarrow \underline{\Gproj }B$ is an equivalence.
Indeed, for any $M\in \underline{\Gproj }B$, there is an isomorphism $M\cong\widetilde{G} \widetilde{F}(M)$
in $\underline{\Gproj }B$ because $M\cong G F (M)$ in $D_{sg}(B)$.
To conclude, we get that both $G$ and $\widetilde{G}$ are triangle equivalences, and then we infer
$D_{def}(A)\cong D_{def}(B)$ by the \cite[Lemma 3.10]{QS23}.
\end{proof}

\section{Applications and examples}
\indent\indent
In this section, we will apply our main result to trivial extensions, Morita rings
and triangular matrix algebras, and we get several reduction methods
on singularity categories and Gorenstein defect categories.

\begin{corollary}\label{cor-trivial-extension}
Let $R$ be an algebra and $T$ be the trivial extension $R\ltimes M$ for a $R$-$R$-bimodule $M$.
Assume that $\pd _{R^e}M<\infty$, $M^{\otimes _Rp}=0$ for some integer $p$ and
$\Tor _i^R(M, M^{\otimes _Rj})=0$ for each $i,j\geq 1$.
Then there are triangle equivalences
$$ D_{sg}(R) \cong D_{sg}(T) \ \ \mbox{and} \ \ D_{def}(R)\cong D_{def}(T).$$
\end{corollary}
\begin{proof}
Clearly, the algebra inclusion $R\hookrightarrow T$ defined by $r\mapsto (r,0)$
has a retraction $T\rightarrow R$ given by $(r,m)\mapsto r$. Therefore, the extension
$R\subseteq T$ splits with $T/R=M$. Now the statement follows from Theorem~\ref{theorem-sing-equiv}
and Theorem~\ref{theorem-def-equiv}.
\end{proof}

Let $B$ and $C$ be algebras, $M$ be a $C$-$B$-bimodule and $A=
\left(\begin{array}{cc} B & 0 \\ M & C  \end{array}\right) $
be a triangular matrix algebra. In \cite{Lu17}, Lu proved that
if $M$ is projective as a $C$-$B$-bimodule, then there are triangle equivalences
$ D_{sg}(A) \cong D_{sg}(B)\coprod  D_{sg}(C)$ and $D_{def}(A)\cong D_{def}(B)\coprod D_{def}(C)$,
see the proof of \cite[Proposition 4.2 and Theorem 4.4]{Lu17}. Now we will generalize
these resulst as follow.

\begin{corollary}\label{cor-tri-matrix-alg}
Keep the above notations.
If $\pd _{C\otimes B^{op}}M<\infty$,
then there are triangle equivalences
$$ D_{sg}(A) \cong D_{sg}(B)\coprod  D_{sg}(C)\ \ \mbox{and} \ \ D_{def}(A)\cong D_{def}(B)\coprod D_{def}(C).$$
\end{corollary}
\begin{proof}
It is clear that $A=
\left(\begin{array}{cc} B & 0 \\ M & C  \end{array}\right) \cong (B\times C)\ltimes M$, $M\mathop{\otimes}\limits_{B\times C}M=0$
and $\Tor _i^{B\times C}(M, M)=0$ for $i\geq 1$. Therefore,
the statement follows from Corollary~\ref{cor-trivial-extension}.
\end{proof}

Let $B$ and $C$ be two algebras, $M$ a $C$-$B$-bimodule,
$N$ a $B$-$C$-bimodule, $\phi :M\otimes _BN\rightarrow C$ a
$C$-$C$-bimodule homomorphism, and $\psi :N\otimes _CM\rightarrow B$ a
$B$-$B$-bimodule homomorphism. From \cite{Bass68,Mor58}, the Morita ring
is of the form $A_{(\phi, \psi)}=
\left(\begin{array}{cc} B & N \\ M & C  \end{array}\right) $, where the addition of elements is
componentwise and the multiplication is given by
$$
\left(\begin{array}{cc} b & n \\ m & c  \end{array}\right) \cdot \left(\begin{array}{cc} b' & n' \\ m' & c'  \end{array}\right)
=\left(\begin{array}{cc} bb'+\psi(n\otimes m') & bn'+nc' \\ mb'+cm' & \phi(m\otimes n')+cc'  \end{array}\right).$$

\begin{corollary}\label{cor-Morita-contex}
 Let $A_{(\phi, \psi)}=
\left(\begin{array}{cc} B & N \\ M & C  \end{array}\right) $
be a Morita ring.
Assuem that $\pd _{C\otimes B^{op}}M<\infty$, $\pd _{B\otimes C^{op}}N<\infty$,
$\Tor _i^B(M,N)=0=\Tor _i^C(N,M)$ for any $i>0$, and either $M\otimes _BN=0$ or $N\otimes _CM=0$.
Then
$$ D_{sg}(A_{(\phi, \psi)}) \cong D_{sg}(B)\coprod  D_{sg}(C).$$ Furthermore, if $\phi=0=\psi$ then
$$D_{def}(A_{(\phi, \psi)})\cong D_{def}(B)\coprod D_{def}(C).$$
\end{corollary}
\begin{proof}
Consider the algebra inclusion $B\times C\hookrightarrow A_{(\phi, \psi)}$ defined by $(b,c)\mapsto
\left(\begin{array}{cc} b & 0 \\ 0 & c  \end{array}\right)$ which is split if $\phi=0=\psi$. Since
$A_{(\phi, \psi)}/(B\times C)=M\oplus N$, it is easy to see that
the statement follows from Theorem~\ref{theorem-sing-equiv}
and Theorem~\ref{theorem-def-equiv}.
\end{proof}

Now we will illustrate our results by an example.
\begin{example}\label{exam-1}
{\rm Let $\Lambda = kQ/I$ be the algebra where $Q$ is the quiver
$$\xymatrix{1 \ar@<+1ex>[r]^\beta  \ar@(ul,dl)_\gamma & 2 \ar@<+1ex>[l]^\alpha }$$
and $I=\langle \gamma ^2, \alpha \beta\rangle$. We write the concatenation of paths from the right to the left.
Let $\Gamma=\Lambda/\Lambda \alpha \Lambda$. Then $\Gamma$ is the quiver
$$\xymatrix{1 \ar[r]^\beta  \ar@(ul,dl)_\gamma & 2  }$$
with relation $\{ \gamma ^2\}$, and $\Gamma$ is a subalgebra of
$\Lambda$. Therefore, $\Lambda \alpha \Lambda$ can be viewed as a $\Gamma$-$\Gamma$
bimodule,
and there is an algebra isomorphism $\Lambda\cong\Gamma \ltimes \Lambda \alpha \Lambda$
mapping $\overline{\varepsilon}_i$ to $(\overline{\varepsilon}_i,0)$, $\overline{\gamma}$ to $(\overline{\gamma},0)$,
$\overline{\beta}$ to $(\overline{\beta},0)$ and $\overline{\alpha}$ to $(0,\overline{\alpha})$,
where $\varepsilon _i$ is the trivial path at $i$. It is easy to check that $\Lambda \alpha \Lambda _\Gamma
\cong (e_2\Gamma/\rad(e_2\Gamma))^4$ and $_\Gamma \Lambda \alpha \Lambda
\cong \Gamma e_1$. Therefore, we get $\Lambda \alpha \Lambda \otimes _\Gamma \Lambda \alpha \Lambda =0$
and $\Tor _i^\Gamma(\Lambda \alpha \Lambda, \Lambda \alpha \Lambda)=0$ for each $i\geq 1$.
Now we claim that $\pd _{\Gamma^e}\Lambda \alpha \Lambda<\infty$, and then
we get triangle equivalences
$ D_{sg}(\Lambda) \cong D_{sg}(\Gamma)$ and $D_{def}(\Lambda)\cong D_{def}(\Gamma)$ by Corollary~\ref{cor-trivial-extension}.
Indeed, the enveloping algebra $\Gamma^e$ has the following quiver
\[\scriptsize \xymatrix@C=1.5cm@R=1.5cm{
  1\times 1^{\mathrm{op}}
  \ar@(ul,dl)_{\gamma\otimes 1^{\mathrm{op}}}
  \ar@(ul,ur)^{1\otimes \gamma^{\mathrm{op}}}
  \ar[r]^{\beta\otimes 1^{\mathrm{op}}}
  \ar@{<-}[d]_{1\otimes\beta^{\mathrm{op}}}
& 2\times 1^{\mathrm{op}}
  \ar@(ul,ur)^{2\otimes\gamma^{\mathrm{op}}}
  \ar@{<-}[d]_{2\otimes\beta^{\mathrm{op}}}
\\1\times 2^{\mathrm{op}}
  \ar[r]_{\beta\otimes 2^{\mathrm{op}}}
  \ar@(ul,dl)_{\gamma\otimes 2^{\mathrm{op}}}
& 2\times 2^{\mathrm{op}} }\]
with relations $\{ (\gamma\otimes 1^{\mathrm{op}})^2, (1\otimes \gamma^{\mathrm{op}})^2,
(2\otimes\gamma^{\mathrm{op}})^2, (\gamma\otimes 2^{\mathrm{op}})^2,
(2\otimes\beta^{\mathrm{op}})(\beta\otimes 2^{\mathrm{op}})-(\beta\otimes 1^{\mathrm{op}})(1\otimes\beta^{\mathrm{op}}),
(\gamma\otimes 1^{\mathrm{op}})(1\otimes\beta^{\mathrm{op}})-(1\otimes\beta^{\mathrm{op}})(\gamma\otimes 2^{\mathrm{op}}),
(2\otimes\gamma^{\mathrm{op}})(\beta\otimes 1^{\mathrm{op}})-(\beta\otimes 1^{\mathrm{op}})(1\otimes \gamma^{\mathrm{op}}),
(\gamma\otimes 1^{\mathrm{op}})(1\otimes \gamma^{\mathrm{op}})-(1\otimes \gamma^{\mathrm{op}})(\gamma\otimes 1^{\mathrm{op}})\}$.
The $\Gamma^e$-module $\Lambda \alpha \Lambda$ is given by the following representation:

\medskip

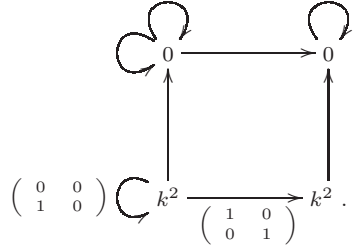
\[\scriptsize \xymatrix@C=1.5cm@R=1.5cm{
  0
  \ar@(ul,dl)
  \ar@(ul,ur)
  \ar[r]
  \ar@{<-}[d]
& 0
  \ar@(ul,ur)
  \ar@{<-}[d]
\\k^2
  \ar[r]_{\tiny {\left(\begin{array}{cc} 1 & 0 \\ 0 & 1  \end{array}\right)}}
  \ar@(ul,dl)_{\tiny { \left(\begin{array}{cc} 0 & 0 \\ 1 & 0  \end{array}\right)}}
& k^2 \ .}\]
There is an exact sequence
$$0\rightarrow \Gamma^ee_{1\times 1^{\mathrm{op}}}\rightarrow \Gamma^ee_{1\times 2^{\mathrm{op}}}
\rightarrow \Lambda \alpha \Lambda \rightarrow 0$$
of $\Gamma^e$-modules, and thus $\pd _{\Gamma^e}\Lambda \alpha \Lambda=1$.

}
\end{example}

\bigskip

\noindent {\footnotesize {\bf Funding} This work is supported by
the National Natural Science Foundation of China (12061060, 11961007),
the project of Young and Middle-aged Academic and Technological leader of Yunnan
(Grant No. 202305AC160005) and
the Scientific and Technological Innovation Team of Yunnan (Grant No.
2020CXTD25).
}


\begin{thebibliography}{99}

\bibitem{AKLY17} L. Angeleri H\"{u}gel, S. K\"{o}nig, Q. Liu and D.
Yang, Ladders and simplicity of derived module categories, J. Algebra 472 (2017), 15--66.

\bibitem{Bass68} H. Bass, Algebraic K-Theory. New York-Amsterdam: W. A. Benjamin, 1968

\bibitem{BJO15} P. A. Bergh, D. A. J{\o}rgensen and S. Oppermann,
The Gorenstein defect category, Q. J. Math 66 (2015), 459--471.

\bibitem{Buch21} R.-O. Buchweitz, Maximal Cohen-Macaulay modules and Tate-cohomology over Gorenstein rings,
with appendices by L.L. Avramov, B. Briggs, S.B. Iyengar, and J.C. Letz, Math. Surveys and
Monographs, vol. 262, Amer. Math. Soc., 2021.

\bibitem{Chen09} X. W. Chen, Singularity categories, Schur functors and triangular matrix rings, Algebr. Represent. Theor. 12 (2009),
181--191.

\bibitem{Chen14} X. W. Chen, Singular equivalences induced by homological epimorphisms, Proc. Am. Math. Soc.
142 (2014), no. 8, 2633--2640.

\bibitem{Chen18} X. W. Chen, The singularity category of a quadratic monomial algebra, Quart. J. Math. 69 (2018), 1015--1033.

\bibitem{CGL15} X. W. Chen, S. F. Geng and M. Lu, The singularity categories of the Cluster-tilted algebras of Dynkin type,
Algebr. Represent. Theory 18 (2015), no. 2, 531--554.

\bibitem{CL15} X. W. Chen and M. Lu, Singularity categories of skewed-gentle algebras, Colloq. Math. 141 (2015), no. 2, 183--198.


\bibitem{CL20} X. W. Chen and M. Lu, Gorenstein homological properties of tensor rings, Nagoya Math. J.
237 (2020), 188--208.

\bibitem{CR20} X. W. Chen and W. Ren, Frobenius functors and Gorenstein homological properties,
J Algebra 610 (2022), 18--37.

\bibitem{CSZ18} X. W. Chen, D. W. Shen, and G. D. Zhou, The Gorenstein-projective modules over a monomial
algebra, Proc. Royal Soc. Edin. A 148 (2018), 1115--1134.

\bibitem{CHQW20} Y. P. Chen, W. Hu, Y. Y. Qin and R. Wang, Singular equivalences and Auslander-Reiten
conjecture, J. Algebra 623 (2023), 42--63.


\bibitem{CLMS21} C. Cibils, M. Lanzilotta, E. N. Marcos and A. Solotar, Jacobi-Zariski long nearly exact sequences for
associative algebras, Bull. Lond. Math. Soc. 53 (2021), no. 6, 1636--1650.

\bibitem{CLMS22} C. Cibils, M. Lanzilotta, E. N. Marcos and A. Solotar, Han's conjecture for
bounded extensions, J. Algebra 598 (2022), 48--67.

\bibitem{Dal21} G. Dalezios, On singular equivalences of Morita type with level and Gorenstein algebras, Bull. Lond.
Math. Soc. 53 (2021), no. 4, 1093--1126.


\bibitem{EPS22} K. Erdmann, C. Psaroudakis and {\O}. Solberg,
Homological invariants of the arrow removal operation, Represent. Theory 26 (2022), 370--387.

\bibitem{Han06} Y. Han, Hochschild (co)homology dimension, J. London Math. Soc. 73 (2006),
657--668.

\bibitem{HP17} W. Hu and S. Pan, Stable functors of derived equivalences and Gorenstein projective
modules, Math. Nachr. 290 (2017), no. 10, 1512--1530.

\bibitem{IM21} K. Iusenko and J. W. MacQuarrie, Homological properties of extensions of
abstract and pseudocompact algebras, ArXiv:2108.12923, 2021.

\bibitem{Kal15} M. Kalck, Singularity categories of gentle algebras, Bull. London Math. Soc. 47 (2015), no. 1, 65--74.

\bibitem{Kato02} Y. Kato, On derived equivalent coherent rings, Comm. Algebra 30 (2002), 4437--4454.

\bibitem{LHZ22} H. H. Li, J. S. Hu and Y. F. Zheng, When the Schur functor induces a triangle-equivalence between Gorenstein defect categories,
Sci China Math 65 (2022), 2019--2034.

\bibitem{Lu17} M. Lu, Gorenstein defect categories of triangular matrix algebras,
J. Algebra 480 (2017), 346--367.

\bibitem{Lu19} M. Lu, Gorenstein properties of simple gluing algebras, Algebr. Represent. Theor. 22 (2019),
517--543.

\bibitem{MN24} J, W. Macquarrie and F. D. R. Naves,
Quotient bifinite extensions and the finitistic dimension conjecture, Proc. Am. Math. Soc. 152 (2024), No. 2, 585--590.

\bibitem{Mor58} K. Morita, Duality for modules and its applications to the theory of rings with minimum condition. Sci Rep Tokyo
Kyoiku Daigaku Sect A, 1958, 6: 83--142

\bibitem{Orl04} D. Orlov, Triangulated categories of singularities and D-branes in Landau-Ginzburg models,
Trudy Steklov Math. Institute 204 (2004), 240--262.

\bibitem{PZ15} S. Y. Pan and X. J. Zhang, Derived equivalences and Cohen-Macaulay Auslander
algebras, Front. Math. China 10 (2015), no. 2, 323--338.

\bibitem{PSS14} C. Psaroudakis, {\O}. Skarts{\ae}terhagen and {\O}. Solberg, Gorenstein categories, singular equivalences
and finite generation of cohomology rings in recollements, Trans. Am. Math. Soc. Ser. B 1 (2014), 45--95.

\bibitem{Qin22} Y. Y. Qin, Reduction techniques of singular equivalences, J. Algebra 612 (2022),
616--635.

\bibitem{Qin24} Y. Y. Qin, A note on singularity categories and triangular matrix algebras,
to appear on Algebr. Represent. Theor., https://doi.org/10.1007/s10468-023-10249-3.

\bibitem{QS23} Y. Y. Qin and D. W. Shen, Eventually homological isomorphisms and Gorenstein projective modules. Sci
China Math, 2023, 66, https://doi.org/10.1007/s11425-022-2146-5.

\bibitem{Rin13} C. M. Ringel, The Gorenstein projective modules for the Nakayama algebras. I, J. Algebra 385
(2013), 241--261.

\bibitem{Shen21} D. W. Shen, A note on singular equivalences and idempotents, Proc. Am. Math. Soc.
149 (2021), no. 10, 4067--4072.

\bibitem{Wang15} Z. F. Wang, Singular equivalence of Morita type with level, J. Algebra 439 (2015) 245--269.

\bibitem{Xi06} C. C. Xi, On the finitistic dimension conjecture. II. Related to finite
global dimension, Adv. Math., 201 (2006), no. 1, 116--142.

\bibitem{XX13} C. C. Xi and D. M. Xu, The finitistic dimension conjecture and
relatively projective modules, Commun. Contemp. Math., 15(2013), no. 2, 1350004, 27.

\bibitem{Z13} P. Zhang, Gorenstein-projective modules and symmetric recollements, J. Algebra 388 (2013), 65--80.


\end{thebibliography}
\end{document}